\documentclass[graybox]{svmult}
\usepackage{mathptmx}       
\usepackage{helvet}         
\usepackage{courier}        
\usepackage{type1cm}        
\usepackage{latexsym}
\usepackage{amsmath}
\usepackage{amssymb}
\usepackage{amscd}
\usepackage{epsfig}
\usepackage{graphicx, subfigure}
\usepackage{makeidx}         
\usepackage{graphicx}        
\usepackage{multicol}        
\usepackage[bottom]{footmisc}

\makeindex             

\begin{document}
\title*{A Linearly-Growing Conversion from the Set Splitting Problem to the Directed Hamiltonian Cycle Problem}
\titlerunning{A Conversion from the Set Splitting Problem to Directed Hamiltonian Cycle Problem}
\author{Michael Haythorpe and Jerzy A. Filar}
\institute{Michael Haythorpe \at Flinders University, Australia, \email{michael.haythorpe@flinders.edu.au}
\and Jerzy A. Filar \at Flinders University, Australia \email{jerzy.filar@flinders.edu.au}}
%
%
\maketitle

\abstract{We consider a direct conversion of the, classical, set splitting problem to the directed Hamiltonian cycle problem. A constructive procedure for such a conversion is given, and it is shown that the input size of the converted instance is a linear function of the input size of the original instance. A proof that the two instances are equivalent is given, and a procedure for identifying a solution to the original instance from a solution of the converted instance is also provided. We conclude with two examples of set splitting problem instances, one with solutions and one without, and display the corresponding instances of the directed Hamiltonian cycle problem, along with a solution in the first example.}

%

\section{Introduction}\label{s:Intro}

The set splitting problem (SSP) is a famous decision problem that can be simply stated: given a finite universe set $\mathcal{U}$, and a family $\mathcal{S}$ of subsets of $\mathcal{U}$, decide whether there exists a partition of $\mathcal{U}$ into two, disjoint, non-empty subsets $\mathcal{U}_1$ and $\mathcal{U}_2$ such that every subset $S^i \in \mathcal{S}$ is {\em split} by this partition. That is, for each subset $S^i \in \mathcal{S}$, we have $S^i \not\subset \mathcal{U}_1$ and $S^i \not\subset \mathcal{U}_2$. If such a partition exists, we call it a {\em solution} of the SSP instance, and say that the decision of the instance is YES. Similarly, if no such partition exists, then the decision of the instance is NO.

This problem has been studied by such distinguished mathematicians as P. Erd\H{o}s \cite{erdos} and E.W. Miller \cite{miller} since the 1930s. Since then, it has been studied by many authors in the mathematics, computer science, and engineering communities. It has acquired a theoretical interest by virtue of its relationship to hypergraph colourability problems (e.g. see Radhakrishnan and Srinivasan \cite{radhakrishnan}). In addition, it has applicability in modern lines of research such as DNA computing (e.g. see Chang et al \cite{chang}), and several recent algorithms for solving SSP have been developed (e.g. see Dehne et al \cite{dehne}, Chen and Lu \cite{chen}, Lokshtanov and Saurabh \cite{lokshtanov}).

SSP is known to be NP-complete \cite{gareyjohnson}. One key feature of NP-complete problems is that an instance of any one NP-complete problem can be converted to an instance of any other NP-complete problem, in such a way that the two instances have the same answer, and the cardinalities of the variables sets in the second instance are polynomial functions of the size of input data for the original instance. The study of NP-complete problems originated with Cook \cite{cook}, who proved that an instance of any problem in the set of NP decision problems can be converted to an equivalent instance of the boolean satisfiability problem (SAT). Therefore, SAT was the first problem proven to be NP-complete. Then, if any NP-complete problem $P_1$ can be {\em converted} to another problem $P_2$, the second problem $P_2$ is also proved to be NP-complete. This is because any instance of $P_2$ can be converted to an instance of SAT (via Cook's theorem), and from there converted (possibly through multiple other problems) to an instance of $P_1$.

Cook's breakthrough approach provided the template for NP-complete conversions, and subsequently it has become commonplace for problems to be converted to SAT. A recent study of this may be seen in Kugele \cite{kugele}. However, there is nothing inherently special about SAT to set it apart from other fundamental NP-complete problems. Motivated by this line of thinking, in this chapter we investigate the conversion of SSP to another fundamental NP-complete problem, namely the directed Hamiltonian cycle problem (HCP). Directed HCP can be described simply: given a graph $\Gamma$ containing a set of vertices $V$, such that $|V| = N$, and a set of directed edges $E$, decide whether there exists a simple cycle of length $N$ in the graph $\Gamma$, or not. Directed HCP was one of the earliest known NP-complete problems \cite{karp} and is a classical graph theory problem which has been the subject of investigation for well over a century. Indeed, a famous instance of HCP - the so-called \lq\lq Knight's tour" problem - was solved by Euler in the 1750s, and it remains an area of active research (e.g. see Eppstein \cite{eppstein}, Borkar et al \cite{borkar}, and Baniasadi et al \cite{baniasadi}).

Arguably, it is interesting to consider what might be called \lq\lq linear orbits" of famous NP-complete problems, such as directed HCP. By this, we mean the set of other NP-complete problems which may be converted to, say, directed HCP in such a way that the input size of the resultant HCP instance is a linear function of the input size of the original problem instance. We refer to such a conversion as a {\em linearly-growing} conversion. Although conversions between NP-complete problems have been extensively explored since 1971, less attention has been paid to the input sizes of the resultant instances after such conversions, and yet input sizes that grow quadratically or higher are likely to produce intractable instances.

In this chapter, we provide a linearly-growing conversion procedure that accepts any instance of SSP as input, and produces an equivalent instance of directed HCP as output. The equivalence is in the sense that a Hamiltonian in the output graph instance supplies a solution to the original SSP instance, and non-Hamiltonicity in the output instance implies infeasibility of the original SSP instance.

\section{Simplifying the SSP instance}

Consider an instance of SSP, containing the universe set $\mathcal{U}$ and the family $\mathcal{S}$ of subsets of $\mathcal{U}$. Before we begin solving the problem, we can attempt to simplify it, to obtain a smaller instance that must still have the same answer as the original. The following steps may be performed:

\begin{enumerate}\item If any $S^i \in \mathcal{S}$ contains only a single entry, then the decision of the SSP instance is NO, as this set cannot be split. In this case there is no need to solve the SSP.
\item If any element $u \in \mathcal{U}$ is not contained in any $S^i \in \mathcal{S}$, then it may be removed from $\mathcal{U}$. This is because $u$ could be placed in either partition without affecting the solution, so it is inconsequential to the problem.
\item If any $S^i \in \mathcal{S}$ is equal to $\mathcal{U}$, then it may be disregarded, as any partitioning of $\mathcal{U}$ into non-empty subsets will split $S^i$.
\item If any $S^i \in \mathcal{S}$ is a subset of some other $S^j \in \mathcal{S}$, then $S^j$ may be disregarded, as any partitioning of $\mathcal{U}$ that splits $S^i$ necessarily splits $S^j$ as well.\end{enumerate}

Once the instance has been simplified in this manner, we say it is in {\em simple form}, and we are ready to begin converting it to an instance of directed HCP.

\section{Algorithm for converting an SSP instance to an instance of directed HCP}\label{sec-construction}

For a given instance $\left<\mathcal{U}, \mathcal{S}\right>$ of SSP we shall construct an instance $\Gamma = \left<V, E\right>$ of HCP possessing the property that any Hamiltonian cycle corresponds, in a natural way, to a solution of the original instance of SSP. Additionally, in the case the constructed graph does not possess a Hamiltonian cycle, neither does the original instance of SSP have a solution.

The algorithm for constructing $\Gamma$ from $\left<\mathcal{U}, \mathcal{S}\right>$ has three main steps in which three sets of vertices and edges are constructed. Collectively, these will comprise the vertex and edge sets of $\Gamma$.

Suppose that we have an instance $\left<\mathcal{U}, \mathcal{S}\right>$ of SSP in simple form. Let $\mathcal{U} = \{1, 2, \hdots, u\}$ denote the universe set, and assume that each $S^i \in \mathcal{S}$ contains entries $s^i_j$ in ascending order. Denote by $s$ the number of subsets remaining after the simplification process, and also define $c := \sum\limits_{i = 1}^s |S^i|$ to be the total number of elements over all subsets $S^i$. Note that $c \geq u$, where the case $c = u$, trivially, has the answer YES. Then we may define an instance of HCP in the form of a graph $\Gamma$ containing vertices $V$ and directed edges $E$, such that the HCP instance is equivalent to the SSP instance, as follows.

The vertex set $V$ can be partitioned into three mutually disjoint subsets of vertices $V^U$, $V^S$ and $V^C$. That is, $V = V^U \cup V^S \cup V^C$. The subset $V^U$ will contain vertices corresponding to each element of the universe set $\mathcal{U}$. The subset $V^S$ will contain vertices corresponding to each subset $S^i \in \mathcal{S}$. The subset $V^C$ will contain two additional \lq\lq connecting" vertices, that will link the $V^U$ and $V^S$ parts to form a cycle.

Likewise, the edge set $E$ can be partitioned into three mutually disjoint subsets of edges $E^U$, $E^S$ and $E^C$. That is, $E = E^U \cup E^S \cup E^C$. The subset $E^U$ will contain edges whose endpoints lie entirely within $V^U$. Similarly, the subset $E^S$ will contain edges whose endpoints lie entirely within $V^S$. Finally, $E^C$ will contain many \lq\lq connecting" edges, which connect vertices from any of the three partitions to other vertices, possibly in different partitions.


\paragraph{{\underline{The conversion algorithm}}}
\paragraph{\underline{Step 1:}}

We will first consider the vertex subset $V^U$ and edge subset $E^U$. These sets can be further partitioned into $u$ subsets, one for each element in the universe set. That is, $V^U = \bigcup\limits_{i=1}^u V^{U,i}$ and $E^U = \bigcup\limits_{i=1}^u E^{U,i}$. Then, for each element $i \in \mathcal{U}$, we can describe $V^{U,i}$ and $E^{U,i}$ directly:

\begin{eqnarray*}V^{U,i} & = & \{v^{U,i}_1, v^{U,i}_2, v^{U,i}_3, v^{U,i}_4\}\\
E^{U,i} & = & \{(v^{U,i}_1,v^{U,i}_2), (v^{U,i}_1,v^{U,i}_3), (v^{U,i}_2,v^{U,i}_3), (v^{U,i}_3,v^{U,i}_2)\}\end{eqnarray*}

Note that at this stage of the construction, $v^{U,i}_4$ is an isolated vertex. Each subgraph $\left<V^{U,i}, E^{U,i}\right>$ may be visualised as in Figure \ref{fig-vU_component}. The thick bold undirected edge between vertices $v^{U,i}_2$ and $v^{U,i}_3$ represents directed edges in both directions between the two vertices.

\begin{figure}[h!]
\centering\includegraphics[scale=0.25]{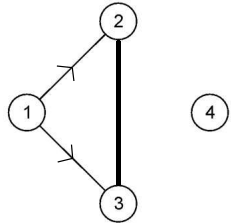}
\label{fig-vU_component}\caption{One $V^{U,i}$ component. Here, vertex $v^{U,i}_j$ is represented simply by label $j$, for the sake of neatness.}
\end{figure}

\paragraph{\underline{Step 2:}}

We will next consider the vertex subset $V^S$ and edge subset $E^S$. These sets can be further partitioned into $s$ subsets, one for each subset $S^i \in \mathcal{S}$. That is, $V^S = \bigcup\limits_{i=1}^s V^{S,i}$ and $E^S = \bigcup\limits_{i=1}^s E^{S,i}$. Then, for each subset $S^i \in \mathcal{S}$, we must first determine $|S^i|$. For neatness, when no confusion is possible we define $k = |S^i|$, taking care to remember that the value of $k$ depends on $i$. Then, the number of vertices in $V^{S,i}$ is chosen to be $5 + 6k$, each of which will be denoted by $v^{S,i}_j$. The edge set $E^{S,i}$ is the union of the following three groups of edges:

Group I:
\begin{eqnarray}(v^{S,i}_1,v^{S,i}_6), (v^{S,i}_1,v^{S,i}_{6+3k}), (v^{S,i}_{5+3k},v^{S,i}_2),\label{edg-s1}\\
(v^{S,i}_{5+6k},v^{S,i}_2), (v^{S,i}_3, v^{S,i}_4), (v^{S,i}_4, v^{S,i}_3), (v^{S,i}_4, v^{S,i}_5), (v^{S,i}_5,v^{S,i}_4),\nonumber\end{eqnarray}

Group II: for all $j = 1, \hdots, k$ (for neatness, we define $a_j = 3+3j$, $b_j = 4+3j$ and $c_j = 5+3j$):

\begin{eqnarray}(v^{S,i}_{a_j}, v^{S,i}_{b_j}), (v^{S,i}_{b_j}, v^{S,i}_{a_j}), (v^{S,i}_{b_j}, v^{S,i}_{c_j}), (v^{S,i}_{c_j},v^{S,i}_{b_j}),\\
(v^{S,i}_{a_j+3k}, v^{S,i}_{b_j+3k}), (v^{S,i}_{b_j+3k}, v^{S,i}_{a_j+3k}), (v^{S,i}_{b_j+3k}, v^{S,i}_{c_j+3k}), \label{edg-s2} (v^{S,i}_{c_j+3k},v^{S,i}_{b_j+3k}),\nonumber\end{eqnarray}

Group III: for all $j = 1, \hdots, k-1$ (retaining the definitions of $a_j$, $b_j$ and $c_j$ from above):

\begin{eqnarray}(v^{S,i}_{c_j},v^{S,i}_{c_j+1}), (v^{S,i}_{c_j},v^{S,i}_{c_j+1+3k}), (v^{S,i}_{c_j+3k}, v^{S,i}_{c_j+1}), (v^{S,i}_{c_j+3k}, v^{S,i}_{c_j+1+3k}),\label{edg-s3}\\
(v^{S,i}_{c_j},v^{S,i}_3), (v^{S,i}_3, v^{S,i}_{c_j+1}), (v^{S,i}_{c_j+3k},v^{S,i}_5), (v^{S,i}_5, v^{S,i}_{c_j+3k+1}).\nonumber\end{eqnarray}

Each subgraph $\left<V^{S,i}, E^{S,i}\right>$ has a characteristic visualisation. In Figure \ref{fig-vS_component} we display such a subgraph for the case where $k = |S^i| = 3$. The thick bold undirected edges represent directed edges in both directions between two vertices. Note that in Figure \ref{fig-vS_component}, the Group I edges are the two directed edges emanating from each of vertex $v^{S,i}_1$ and $v^{S,i}_2$, as well as the undirected edges between vertices $v^{S,i}_3$, $v^{S,i}_4$ and $v^{S,i}_5$. The Group II edges are the undirected edges on the top and bottom of the figure. The Group III edges are all of the directed edges in the interior of the figure.

\begin{figure}[h!]
\centering\includegraphics[scale=0.4]{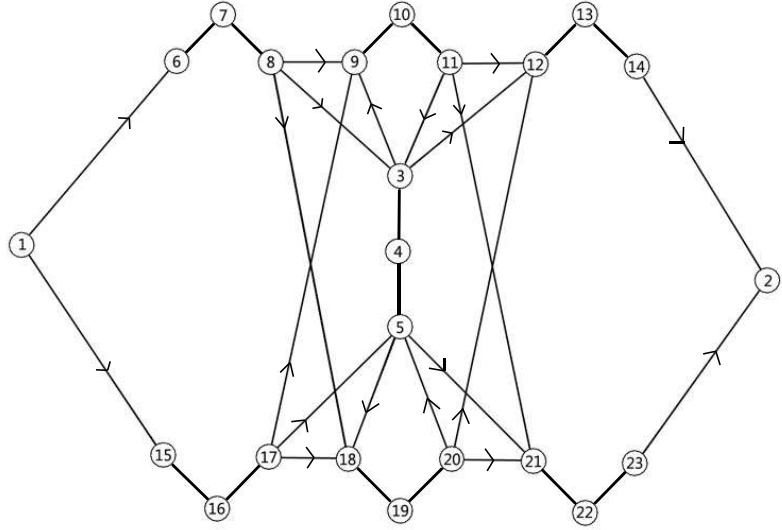}
\label{fig-vS_component}\caption{One $V^{S,i}$ component. Here, vertex $v^{S,i}_j$ is represented simply by label $j$, for the sake of neatness.}
\end{figure}

\paragraph{\underline{Step 3:}}

Finally, we consider the vertex subset $V^C$ and edge subset $E^C$. There are only two vertices in $V^C$, namely $v^C_1$ and $v^C_2$. However, there are many edges in $E^C$, and a procedure must be undertaken to identify them all. First, we include the following edges in $E^C$:

\begin{eqnarray}(v^{U,u}_4, v^C_1), (v^C_1, v^{S,1}_1), (v^{S,s}_{5+6|S^s|}, v^C_2), (v^C_2, v^{U,1}_1),\label{edg-c1}\end{eqnarray}

as well as the following edges for each $i = 1, \hdots, u-1$:

\begin{eqnarray}(v^{U,i}_4, v^{U,i+1}_1),\label{edg-c2}\end{eqnarray}

and also the following edges for each $j = 1, \hdots, s-1$:

\begin{eqnarray}(v^{S,j}_{5+6|S^j|}, v^{S,j+1}_1).\label{edg-c3}\end{eqnarray}

The edges in (\ref{edg-c1})--(\ref{edg-c3}) link the various components of the graph together. Specifically, the first group of edges links the $V^U$ component to the $V^S$ component, the second group links each $V^{U,i}$ component with the $V^{U,i+1}$ component that follows it, and the third group links each $V^{S,i}$ component with the $V^{S,i+1}$ component that follows it. At this stage of construction, the graph $\Gamma$ can be visualised as in Figure \ref{fig-partial_conversion}.

\begin{figure}[h!]
\centering\includegraphics[scale=0.15]{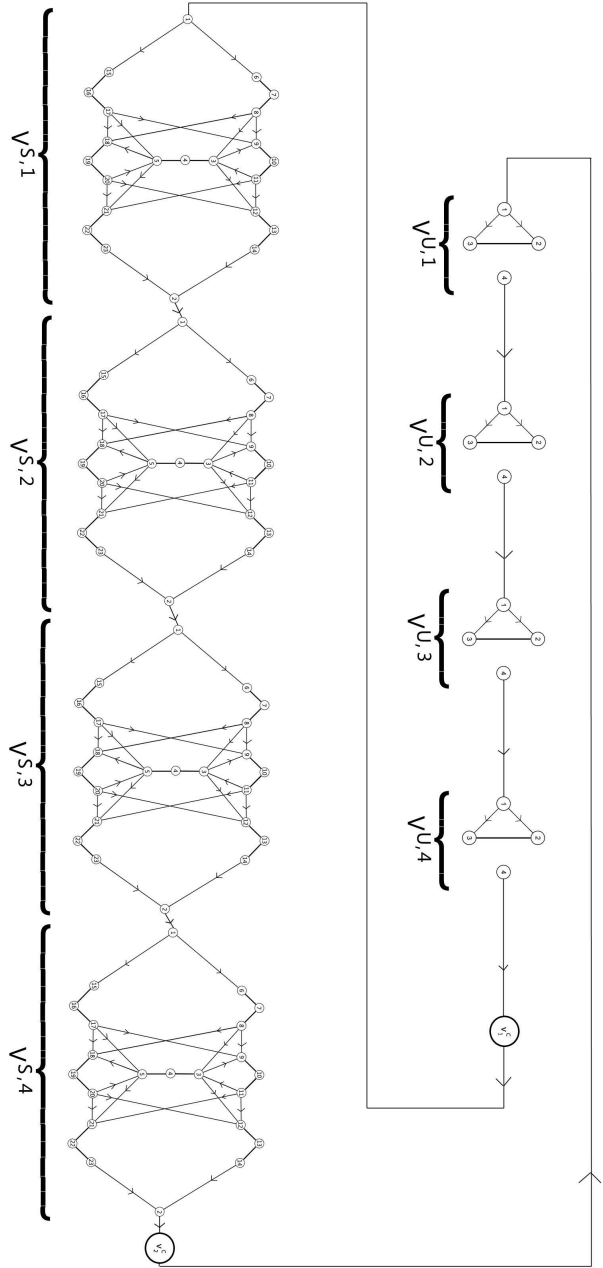}
\caption{The graph $\Gamma$ at an intermediate stage of construction. Note that in this example, $u = s = 4$ and all $|S^i| = 3$.\label{fig-partial_conversion}}
\end{figure}

However, the above edges do not comprise all of $E^C$. Additional edges need to be added, as follows. For each $i \in \mathcal{U}$, we undertake the following procedure to insert additional edges in $E^C$. First we identify all subsets in $\mathcal{S}$ which contain the element $i$, and store their indices in a set $F^i$. We also record a set $R^i$, which contains the positions of element $i$ in each subset $S^{F^i_j}$. For example, suppose that the subsets are $S^1 = (1,3,6)$, $S^2 = (2,3,4)$, $S^3 = (2,4,6)$ and $S^4 = (1,2,5)$. Then $F^1 = (1,4)$, and $R^1 = (1,1)$. Similarly, $F^2 = (2,3,4)$, and $R^2 = (1,1,2)$. For the sake of neatness, when no confusion is possible, we will define $f = |F^i|$, taking care to remember that $f$ depends on the value of $i$.

For each $i \in \mathcal{U}$, we define $d_{ij} = 3 + 3R^i_j$, and $e_{ij} = 3 + 3|S^{F^i_j}| + 3R^i_j$, and insert the following edges:

\begin{eqnarray}(v^{U,i}_2, v^{S,F^i_1}_{2+e_{i1}}), (v^{S,F^i_f}_{e_{if}}, v^{U,i}_4), (v^{U,i}_3, v^{S,F^i_1}_{2+d_{i1}}), (v^{S,F^i_f}_{d_{if}}, v^{U,i}_4),\label{edg-c4}\end{eqnarray}

into $E^C$. Finally for $i \in \mathcal{U}$, and each $j = 1, \hdots, f-1$ (retaining the definitions of $f$, $d_{ij}$ and $e_{ij}$ from above), we insert pairs of edges:

\begin{eqnarray}(v^{S,F^i_j}_{e_{ij}}, v^{S,F^i_{j+1}}_{2+e_{i,j+1}}), (v^{S,F^i_j}_{d_{ij}}, v^{S,F^i_{j+1}}_{2+d_{i,j+1}}),\label{edg-c5}\end{eqnarray}

into $E^C$. The edges in (\ref{edg-c4}) and (\ref{edg-c5}) have the effect of creating two paths, that each travel from one of the vertices in $V^{U,i}$ (either $v^{U,i}_2$ or $v^{U,i}_3$), through three vertices in each $V^{S,F^i_j}$, and finally return to $V^{U,i}$, specifically to the vertex $v^{U,i}_4$. Two such paths are illustrated in Figure \ref{fig-finished_conversion}, for $i = 2$. In the example shown in Figure \ref{fig-finished_conversion}, we assume $S^1 = (1,2,3)$, $S^2 = (1,2,4)$, $S^3 = (1,3,4)$ and $S^4 = (2,3,4)$. The completed graph would have six more paths, two for each of $i = 1$, $i = 3$ and $i = 4$, creating a connected graph. The latter have been omitted for the sake of visual clarity.

\begin{figure}[h!]
\centering\includegraphics[scale=0.15]{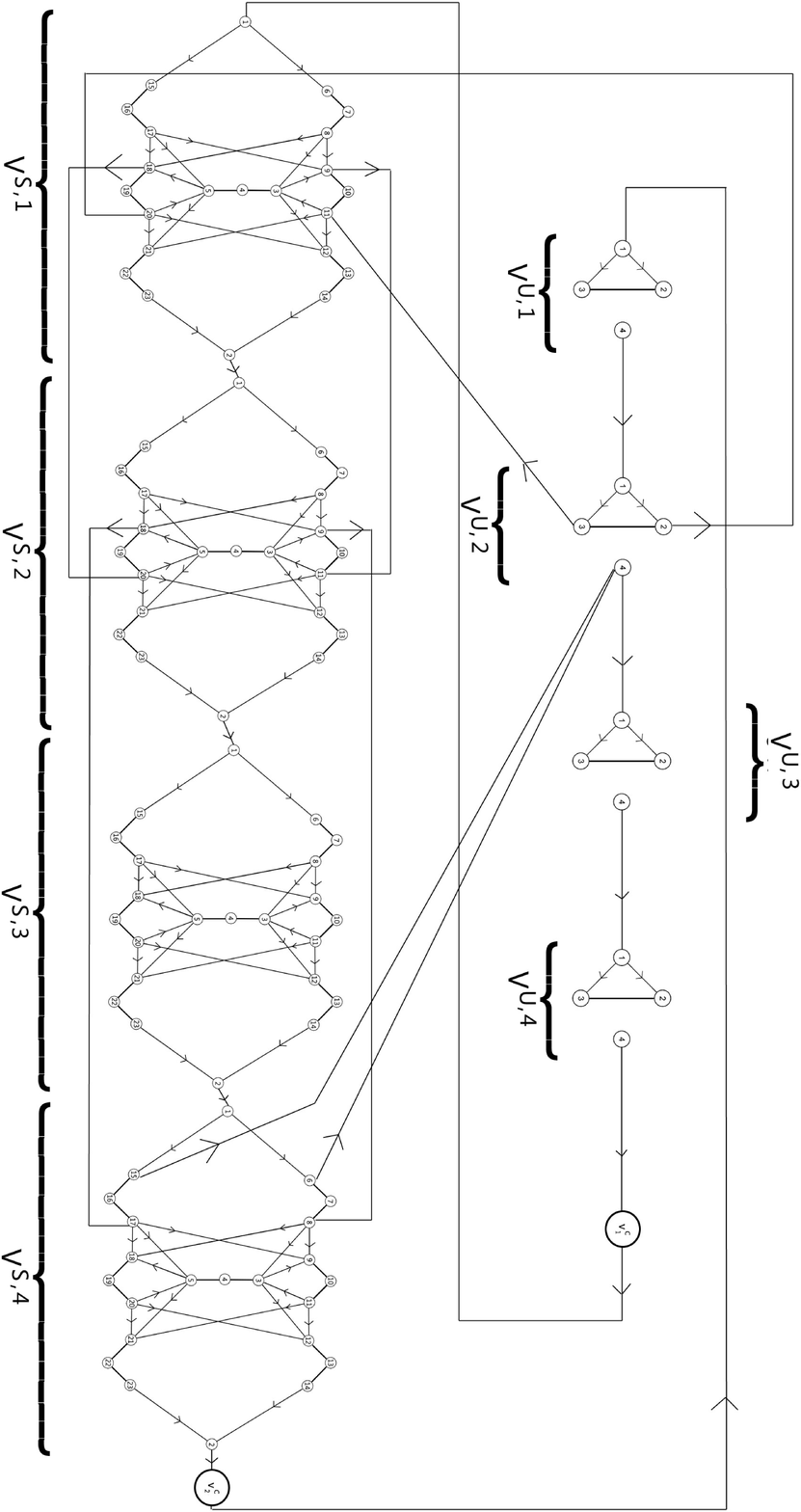}
\caption{The graph $\Gamma$ after being fully constructed. Note that, for the sake of clarity, we only show the new paths corresponding to the element 2. In this example, $S^1 = (1,2,3)$, $S^2 = (1,2,4)$, $S^3 = (1,3,4)$ and $S^4 = (2,3,4)$.\label{fig-finished_conversion}}
\end{figure}

This completes the construction of $\Gamma = \left<V, E\right>$ from $\left<\mathcal{U}, \mathcal{S}\right>$. We shall now calculate the cardinalities of $V$ and $E$.

\paragraph{Dimensionality of the constructed graph}

The final graph $\Gamma$will contain 4 vertices for each $i \in \mathcal{U}$, 5 vertices for each $S^i \in \mathcal{S}$, 6 vertices for each entry $s^i_j$, and two additional vertices $v^C_1$ and $v^\mathcal{C}_2$. Therefore the total number of vertices in the graph is $4u + 5s + 6c + 2$.

Counting the number of edges takes a bit more work. There are 4 edges (counting undirected edges as two, directed, edges) in each $E^{U,i}$. So $E^U$ contributes $4u$ edges.

Then for each $E^{S,i}$, there are 8 edges that will always be present (two from $v^{S,i}_1$, two going to $v^{S,i}_2$, and four between $v^{S,i}_3$, $v^{S,i}_4$ and $v^{S,i}_5$. Then for each element $s^i_j$ there are 16 additional edges, except for the final element $s^i_{|S^i|}$ for which there are eight edges. So $E^S$ contributes $8s + 16(c-s) + 8s = 16c$ edges.

Finally, for $E^C$, there are $u$ connecting edges for the $V^{U,i}$ components, and $s$ connecting edges for the $V^{S,i}$ components. There are two more connecting edges emerging from $v^C_1$ and $v^C_2$. Finally, for each element $i \in \mathcal{U}$, there are $2|F_i|+2$ connecting edges forming the two paths, where $|F_i|$ is the number of subsets containing element $i$. So $E^C$ contributes $u + s + 2 + 2(c + u) = 3u + s + 2c + 2$.

Therefore, the total number of directed edges in the graph is $7u + s + 18c + 2$. It should be noted that both the cardinality of the vertex set $V$ and edge set $E$ are linear functions of the cardinalities of the input sets for the original problem. For this reason, we refer to the conversion described above as a {\em linearly-growing conversion} from SSP to directed HCP.

\section{The intended interpretation of the converted graph}

Once the conversion is complete, the graph $\Gamma$ contains many components $V^{U,i}$ and $V^{S,i}$ corresponding to elements in $\mathcal{U}$ and subsets $S^i \in \mathcal{S}$, respectively. We now consider in some detail the intended interpretation of any Hamiltonian cycle traversing those components.

For each element $\hat{u} \in \mathcal{U}$, there are four vertices in $V^{U,\hat{u}}$. As will be proved later, any Hamiltonian cycle in $\Gamma$ visits $v^{U,\hat{u}}_2$ and $v^{U,\hat{u}}_3$ in succession (in either order). If $v^{U,\hat{u}}_2$ is visited last, it corresponds to placing the element $\hat{u}$ in the first partition. If vertex $v^{U,\hat{u}}_3$ is visited last, it corresponds to placing the element $\hat{u}$ in the second partition.

For each $S^i \in \mathcal{S}$, and $k := |S^i|$, by construction, there are $5 + 6k$ vertices in $V^{S,i}$. Now, consider a particular element $s^i_j$, that is, the $j$-th entry of $S^i$. Naturally, $s^i_j \in \mathcal{U}$. Then, corresponding to this element being chosen in the first partition, there are vertices $v^{S,i}_{3+3j}$, $v^{S,i}_{4+3j}$ and $v^{S,i}_{5+3j}$. Similarly, corresponding to this element being chosen in the second partition, there are vertices $v^{S,i}_{3+3k+3j}$, $v^{S,i}_{4+3k+3j}$ and $v^{S,i}_{5+3k+3j}$.

Now, suppose there is an element $\hat{u}$ that we have chosen to be in $\mathcal{U}_1$. That is, the cycle visits $v^{U,\hat{u}}_2$ after $v^{U,\hat{u}}_3$. Consider all subsets $S^i$ that contain $\hat{u}$, that is $s^i_{j_i} = \hat{u}$ for some $j_i$. Then, immediately after visiting $v^{U,\hat{u}}_2$, the cycle will begin to travel to each such component $V^{S,i}$, and in each of them it will traverse the three vertices corresponding to selecting $s^i{j_i}$ in $\mathcal{U}_2$ (not $\mathcal{U}_1$). In this way, only those vertices corresponding to this element being chosen in the correct partition will remain to be traversed later. Once all such vertices in all relevant $V^{S,i}$ components have been visited, the cycle returns to the $V^{U,\hat{u}}$ component in order to proceed to make a choice of partition for the next element in $\mathcal{U}$.

The process until this point will be called \lq\lq stage 1". Once a choice of partition has been made for all elements in $\mathcal{U}$, the cycle travels through $v^C_1$, and on to $v^{S,1}_1$, whereby \lq\lq stage 2" begins.

The intention is that, upon reaching vertex $v^{S,i}_1$ for each $i$, the cycle will recognise which partition $s^i_1$ is in, because the vertices corresponding to the alternative partition will have already been visited during stage 1. The cycle will then proceed through the three vertices corresponding to the correct choice of partition. Then, the cycle will again recognise which partition $s^i_2$ is in, traverse the three vertices corresponding to that choice, and so on. Once the vertices corresponding to all elements, in their partitions, of $S^i$ are traversed, the cycle will reach vertex $v^{S,i}_2$ and proceed to the next component $V^{S,i+1}$. Specifically, it will travel from $v^{S,i}_2$ to $v^{S,i+1}_1$.

During this process, however, the option to visit one of $v^{S,i}_3$ or $v^{S,i}_5$ will present itself each time the the three vertices for an element and partition are traversed. Specifically, after traversing through vertices corresponding to an element being in $\mathcal{U}_1$, it will be possible to visit $v^{S,i}_3$. Alternatively, after traversing through vertices corresponding to an element being in $\mathcal{U}_1$, it will be possible to visit $v^{S,i}_5$. If the cycle chooses to visit one of these two vertices, it will continue through $v^{S,i}_4$ and out to the other of the two vertices. At this point the cycle will continue to traverse the three vertices corresponding to the next element in the set, but it will be forced to choose vertices corresponding to the opposite partition to that of the previous element. For example, suppose the cycle traverses the three vertices corresponding to $s^i_j$ being chosen in the first partition, and then visits $v^{S,i}_3$. It will subsequently visit $v^{S,i}_4$ and $v^{S,i}_5$, and then proceed to visit the three vertices corresponding to $s^i_{j+1}$ being chosen in the second partition.

The above paragraph describes the essence of why the conversion works. Vertices $v^{S,i}_3$, $v^{S,i}_4$ and $v^{S,i}_5$ must be traversed at some point during stage 2, but after they are traversed, the element that is subsequently considered must be in a different partition to the previous element. This will be possible in all $V^{S,i}$ if and only if the partition choices made in stage 1 split all of the subsets. Therefore, as will be shown rigorously below), Hamiltonian cycles will only exist if it is possible to solve the original instance of SSP.

The formal proof of the validity of the above description is presented next.

\section{Proof of conversion}

In this section we will prove that, for an instance of SSP, the conversion given in Section \ref{sec-construction} produces a graph that possesses the desired properties; that is, the graph is Hamiltonian if and only if the original instance of SSP has solutions, that all Hamiltonian cycles correspond to solutions of the instance of SSP, and that all such solutions can be recovered from the Hamiltonian cycles in polynomial time. Throughout the proof we will take advantage of the structure produced the graph construction. We now outline the primary such advantages, before continuing with the main result.

While attempting to construct a Hamiltonian cycle in the converted graph, we will regularly take advantage of {\em forced} choices. These fall into three categories. If a vertex is arrived at, and only one outgoing edge $e$ exists that leads to an as-of-yet unvisited vertex, the cycle will be forced to proceed along the edge $e$. In such a case, we say there is only {\em one remaining outgoing edge}. Alternatively, if a vertex is arrived at, and there is an outgoing edge $e$ that leads to another vertex, for which no other incoming edges from as-of-yet unvisited vertices exist, the cycle will be forced to proceed along the edge $e$. In such a case, we say there is only {\em one remaining incoming edge}. Finally, if a vertex is arrived at, and there is an outgoing edge $e$ that leads to an as-of-yet unvisited vertex $v$ which has degree 2, the cycle will be forced to proceed along edge $e$. By degree 2, we mean that $v$ has exactly two outgoing edges that lead to vertices $v_2$ and $v_3$, and exactly two incoming edges that also come from vertices $v_2$ and $v_3$. Note that this is a slightly non-standard, but convenient, use of the word \lq\lq degree" in a directed graph.

Suppose that, during in the process of constructing a Hamiltonian cycle, we arrive at a vertex $v^{U,i}_1$. There are two possible choices of which vertex to visit next: $v^{U,i}_2$ and $v^{U,i}_3$. Whichever choice is made, we are forced to visit the other immediately afterwards, as there will only be one remaining incoming edge. At this point, regardless of which vertex we have arrived at, there will only be one remaining outgoing edge, which will lead to a vertex in one of the $V^{S,j}$ components, specifically a vertex $v^{S,j}_{5+3k}$ for some positive integers $j$ and $k$, where $5 + 3k \leq |V^{S,j}|$. We will refer to this situation as a {\em type 1 forced path}.

Also, suppose that, during the process of constructing a Hamiltonian cycle, we travel from a vertex in any component other than $V^{S,j}$, to a vertex $v^{S,j}_{5+3k}$ for some positive integers $j$ and $k$. Note that this vertex is adjacent to a degree 2 vertex $v^{S,j}_{4+3k}$. Since the vertex $v^{S,j}_{4+3k}$ was not visited immediately before $v^{S,j}_{5+3k}$, we are forced to visit it immediately, and then proceed along the one remaining outgoing edge to the vertex $v^{S,j}_{3+3k}$. At this point, there will also be only one remaining outgoing edge, which will either lead to a vertex of the form $v^{S,l}_{5+3m}$ for some positive integers $l$ and $m$, where $l > j$ and $5 + 3m \leq |V^{S,l}|$, or it will lead to the vertex $v^{U,i}_4$. Note that in the former case, we arrive in the same type of situation that we started with at the beginning of this paragraph. In the latter case, however, we arrive at $v^{U,i}_4$ and are then forced to proceed either to $v^{U,i+1}_1$ (if $i < u$), or to $v^C_1$ (if $i = u$). We will refer to this situation as a {\em type 2 forced path}.

We now pose the main result of this chapter.

\begin{proposition}\label{prop-conversion_works}Consider an instance of SSP with universe set $\mathcal{U}$ and a family of subsets $\mathcal{S}$, and the graph $\Gamma = \left<V,E\right>$ constructed as in Section \ref{sec-construction}. Then the following three properties hold:

\begin{enumerate}
\item[(i)] If no partition of $\mathcal{U}$ exists that splits all $S^i \in \mathcal{S}$, then $\Gamma$ is non-Hamiltonian.
\item[(ii)] If a partition of $\mathcal{U}$ exists that does split all $S^i \in \mathcal{S}$, a corresponding Hamiltonian cycle exists in $\Gamma$.
\item[(iii)] From any Hamiltonian cycle in $\Gamma$ we can, in polynomial time, identify a corresponding partition of $\mathcal{U}$ that constitutes a solution of the original instance of SSP.
\end{enumerate}\end{proposition}

\begin{proof}

\paragraph{\underline{Stage 1:}}

Suppose now that we attempt to construct a Hamiltonian cycle in $\Gamma$. Since we may begin at any vertex, we choose to begin at the vertex $v^{U,1}_1$. As described above, we undergo a type 1 forced path, and eventually depart from either $v^{U,1}_2$ or $v^{U,1}_3$ and arrive at the vertex $v^{S,i}_{5+3j}$ for some $i$ and $j$. Then, since we have arrived from a component other than $V^{S,i}$, we will undergo a type 2 forced path. Then, we may (or may not) arrive at another vertex for which a type 2 forced path is applicable. Inductively, the process continues until we do not arrive at such a vertex\footnote{In fact, due to the nature of the construction, if the element appears in $q$ subsets, then exactly $q$ type 2 forced paths must occur here.}. Throughout the process, we visit all of the vertices that correspond to placing element 1 into a particular member of the partition. The construction is such that visiting $v^{U,1}_2$ after $v^{U,1}_3$ forces us to visit the vertices corresponding to the element 1 being in $\mathcal{U}_2$. Similarly, visiting $v^{U,1}_3$ after $v^{U,1}_3$ forces us to visit the vertices corresponding to the element 1 being in $\mathcal{U}_1$. Once all of these vertices are visited, we travel to the vertex $v^{U,1}_4$ and proceed to the vertex $v^{U,2}_1$. Note that at this point, all vertices in the $V^{U,1}$ component have been visited.

The above process is then repeated for all components $V^{U,i}$. The only choice that is made in each component is whether to traverse through $v^{U,i}_1 \rightarrow v^{U,i}_2 \rightarrow v^{U,i}_3$, or to traverse through $v^{U,i}_1 \rightarrow v^{U,i}_3 \rightarrow v^{U,i}_2$. After this choice has been made, the path that is taken - through all vertices corresponding to the element $i$ in the opposite member of the partition, in all components corresponding to the subsets in which element $i$ appears - is forced until the next component $v^{U,i+1}$ is reached. Eventually, after the entirety of $V^U$ has been traversed, vertex $v^C_1$ is reached, and we are forced to proceed to the vertex $v^{S,1}_1$.

\paragraph{\underline{Stage 2:}}

There are two outgoing edges from $v^{S,1}_1$, leading to vertices $v^{S,1}_6$ and $v^{S,1}_{6+3k}$ respectively, where $k := |S^1|$. However, exactly one of these must have been visited already in stage 1. If element 1 was placed in $\mathcal{U}_1$, then vertex $v^{S,1}_{6+3k}$ will have already been visited, or similarly, if element 1 was placed in $\mathcal{U}_2$, then vertex $v^{S,1}_6$ will have already been visited. So the choice at this stage is forced. Then, both vertices $v^{S,1}_6$ and $v^{S,1}_{6+3k}$ are adjacent to degree 2 vertices, so the next two steps are forced as well. This means we visit all three vertices corresponding to the element $s^1_1$ being placed in the member of the partition that was chosen during stage 1.

At this point, there are two choices. We may either continue onto the vertices corresponding to the element $s^1_2$, or we may visit one of $v^{S,1}_3$ or $v^{S,1}_5$ (depending on whether $s^1_1$ was placed in $\mathcal{U}_1$ or $\mathcal{U}_2$, respectively). If we make the former choice, we repeat the above process for the element $s^1_2$, and end up again having to choose whether to visit the vertices corresponding to the element $s^1_3$, or visit one of $v^{S,1}_3$ or $v^{S,1}_5$. Suppose that we make the former choice for the first $j-1$ elements of $S^1$, and then choose the latter for the $j$-th element. Without loss of generality, suppose the element $s^1_j$ was chosen, during stage 1, to be in $\mathcal{U}_1$. So, after traversing the vertices corresponding to $s^1_j$, we then choose to visit $v^{S,1}_3$. Then there is an adjacent degree 2 vertex, $v^{S,1}_4$, so we are forced to travel there next, and then to $v^{S,1}_5$.

At this point, our choice is forced as the first of the vertices corresponding to choosing $s^1_{j+1}$ in $\mathcal{U}_2$ has at most one remaining incoming edge, from $v^{S,1}_5$. Therefore this choice must be made, if possible. If it is not possible (because element $s^1_2$ was chosen in $\mathcal{U}_1$ during stage 1), then the Hamiltonian cycle cannot be completed at this point. In such a case, choosing to visit vertex $v^{S,1}_3$ after traversing the vertices corresponding to the first $j$ elements was an incorrect choice. It is then clear that this choice may only be made if $s^1_{j+1}$ was chosen in $\mathcal{U}_2$ during stage 1.

An equivalent argument to that in the above paragraph can be made if element $s^1_j$ was chosen to be in $\mathcal{U}_2$. Then we can see that we may only choose to visit $v^{S,1}_3$ (or $v^{S,1}_5$) when elements $s^1_j$ and $s^1_{j+1}$ are in opposite members of the partition.

After we have visited vertices $v^{S,1}_3$, $v^{S,1}_4$ and $v^{S,1}_5$ once, they cannot be visited again, and so the remaining path through $V^{S,1}$ is forced until we finally reach $v^{S,1}_2$ and are forced to continue to $v^{S,2}_1$. The same process then continues for all components $V^{S,i}$. Finally, the vertex $v^C_2$ is visited, and we travel back to the vertex $v^{U,1}_1$ to complete the Hamiltonian cycle.

The only way in which the above process might fail to produce a Hamiltonian cycle is if there is a component $V^{S,i}$ for which we are unable to find a $j$ such that $s^i_j$ and $s^i_{j+1}$ are in opposite members of the partition. In this case, following the above argument, vertices $v^{S,i}_3$, $v^{S,i}_4$ and $v^{S,i}_5$ cannot possibly be visited in a Hamiltonian cycle. This situation arises only when all entries in $S^i$ are contained in a single member of the partition. In such a situation, the partitioning choices in stage 1 do not form a solution to the original set splitting problem. So making partition choices in stage 1 that do not solve the instance of SSP will always make it impossible to complete stage 2. Clearly then, if there is no solution to the set splitting problem, it will be impossible to find any Hamiltonian cycle. Therefore part (i) holds.

If there are solutions to the original set splitting problem, then for any such solution we can make the corresponding choices in stage 1. Then, in stage 2 there will be an opportunity in each $V^{S,i}$ component to visit the vertices $v^{S,i}_3$, $v^{S,i}_4$ and $v^{S,i}_5$, and continue afterwards. Therefore, for any such solution, a Hamiltonian cycle exists, and hence part (ii) holds.

Finally, identifying the solution to the original instance of SSP is as easy as looking at $v^{U,i}_2$ and $v^{U,i}_3$ for each $i \in \mathcal{U}$ to see which vertex was visited last on the Hamiltonian cycle. If vertex $v^{U,i}_2$ was visited last, element $i$ should be placed in $\mathcal{U}_1$. If vertex $v^{U,i}_3$ was visited last, element $i$ should be placed in $\mathcal{U}_2$. This process can obviously be performed in polynomial time, and therefore part (iii) holds.
\end{proof}

Proposition \ref{prop-conversion_works} ensures that the conversion given in Section \ref{sec-construction} produces a graph which is Hamiltonian if and only if the original set splitting problem has solutions, and each Hamiltonian cycle specifies one such solution. Since the number of vertices and edges in the resultant graph are linear functions of the original problem variables, this process constitutes a linearly-growing NP-complete conversion.

\section{Examples}

We now conclude with two small examples of SSP instances and their corresponding directed HCP instances. In the first example we provide a Hamiltonian cycle in the graph, and hence deduce a solution to the original SSP instance. In the second example, the corresponding graph is non-Hamiltonian, and we deduce that the original SSP instance has no solutions and hence the decision of the instance is NO.

\begin{example}\label{ex-ssp}Consider the SSP instance with $\mathcal{U} = \{1, 2, 3, 4\}$ and a family of subsets $\mathcal{S} = \{S^1,S^2\}$ where $S^1 = \{1, 2, 3\}$ and $S^2 = \{2, 4\}$.

Following the construction given in Section \ref{sec-construction}, we obtain an instance $\Gamma$ of directed HCP, which is displayed in Figure \ref{fig-example}. A Hamiltonian cycle in $\Gamma$ is also displayed in Figure \ref{fig-example}, with the edges in the Hamiltonian cycle designated by dashed or dotted edges. The dashed edges correspond to edges which are chosen in stage 1, while the dotted edges correspond to edges which are chosen in stage 2.

\begin{figure}[h!]
\centering\includegraphics[scale=0.375]{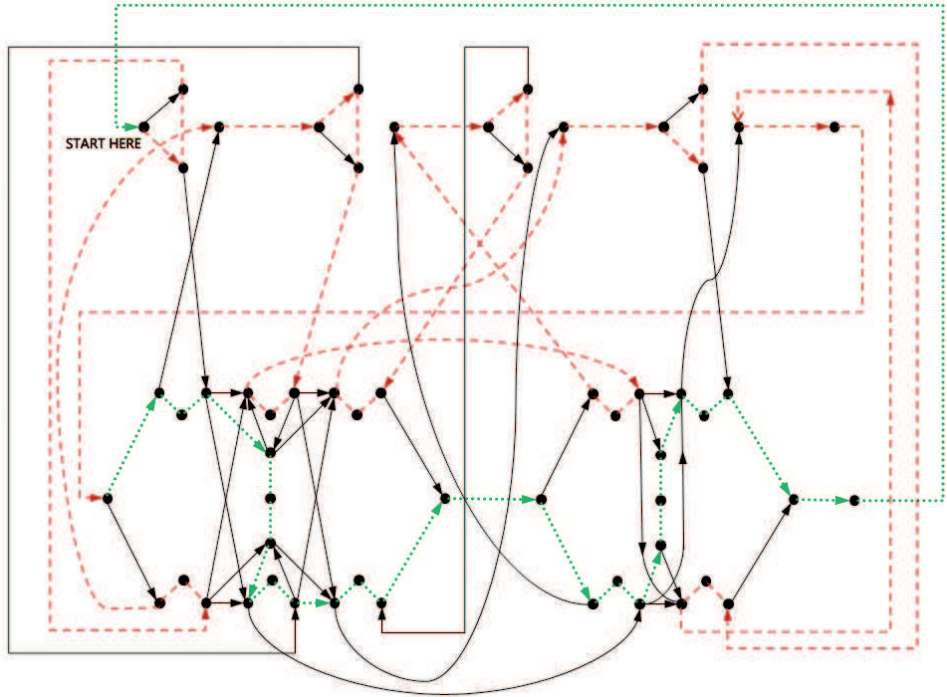}
\caption{The converted graph $\Gamma$ arising from the SSP instance in Example \ref{ex-ssp}. The dashed and dotted edges correspond to the Hamiltonian cycle, with the dashed edges being chosen in stage 1, and the dotted edges being chosen in stage 2.\label{fig-example}}
\end{figure}

Consider the Hamiltonian cycle indicated in Figure \ref{fig-example}, and traverse it from the starting vertex $v^{U,1}_1$ as marked on the figure. Then, to determine the solution of the instance of SSP, we simply need to look at which of vertices $v^{U,i}_2$ and $v^{U,i}_3$ were visited last, for each $i$, on the Hamiltonian cycle. In Figure \ref{fig-example} vertex $v^{U,i}_2$ is drawn above vertex $v^{U,i}_3$ in each case. It can be easily checked that vertices $v^{U,1}_2$, $v^{U,2}_3$, $v^{U,3}_3$ and $v^{U,4}_2$ are the last visited vertices in each case. This corresponds to a solution of the instance of SSP where elements $1$ and $4$ are placed in $\mathcal{U}_1$, and elements $2$ and $3$ are placed in $\mathcal{U}_2$. Clearly, this choice provides a splitting of $S^1$ and $S^2$.


We now check that an incorrect choice of partitioning will make it impossible to complete a Hamiltonian cycle. Suppose we assign elements $1$ and $3$ to $\mathcal{U}_1$, and elements $2$ and $4$ to $\mathcal{U}_2$. Note that this is not a solution of the original instance of SSP, as the subset $S^2$ is not split. In Figure \ref{fig-example2} we use dashed edges to denote the edges visited in stage 1 (corresponding to our incorrect choice of partitioning), and dotted edges to denote the edges subsequently visited in stage 2. However, the three middle vertices in the $V^{S,2}$ component are unable to be visited, and hence this choice of partitioning can not lead to a Hamiltonian cycle.

\begin{figure}[h!]
\centering\includegraphics[scale=0.375]{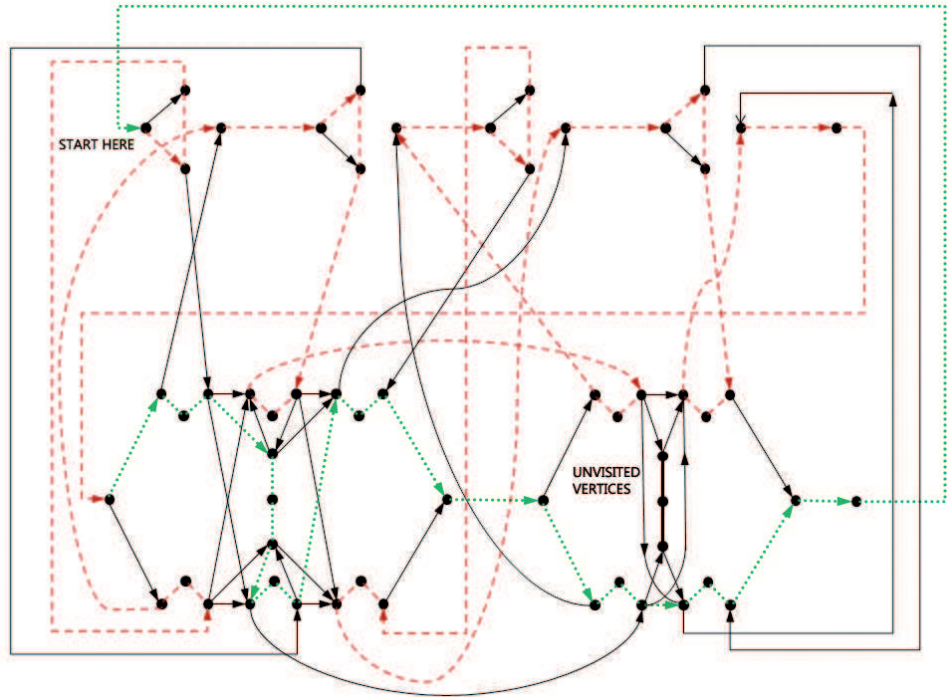}
\caption{The converted graph $\Gamma$ as in Figure \ref{fig-example}, but with a choice of edges in stage 1 that do not correspond to a solution. The three middle vertices in the bottom-right component cannot be visited during stage 2.\label{fig-example2}}
\end{figure}
\end{example}

Finally, we consider an SSP instance with no solutions, and the corresponding instance of directed HCP.

\begin{example}\label{ex-ssp2}Consider the SSP instance with $\mathcal{U} = \{1, 2, 3\}$ and a family of subsets $\mathcal{S} = \{S^1, S^2, S^3\}$ where $S^1 = \{1, 2\}$, $S^2 = \{1, 3\}$ and $S^3 = \{2, 3\}$. It is clear that this instance has no solution since, without loss of generality, if element is placed in the first partition, elements 2 and 3 must be placed in the second partition (to split the first two subsets), and then the third subset is not split.

Following the construction given in Section \ref{sec-construction}, we obtain an instance $\Gamma_2$ of directed HCP, which is displayed in Figure \ref{fig-example3}. Although it is not obvious at first glance, the HCP instance is, indeed, non-Hamiltonian. We have confirmed its non-Hamiltonicity using the Concorde TSP Solver \cite{concorde}.

Since $\Gamma$ is non-Hamiltonian, by Proposition \ref{prop-conversion_works}(i) there is no partition $(\mathcal{U}_1, \mathcal{U}_2)$ which splits all $S^i \in \mathcal{S}$. Suppose we naively tried to create such a splitting by assigning, for example, element 1 to the first partition, and elements 2 and 3 to the second partition. This would correspond to attempting to find a Hamiltonian cycle in $\Gamma_2$ containing the six highlighted edges in Figure 7. Then it is straightforward to verify that it is impossible to complete a Hamiltonian cycle. This is a similar feature to that exhibited in Figure \ref{fig-example2}, with the important exception that in the latter there were alternative, correct, choices of partitioning $\mathcal{U}$ which do permit Hamiltonian cycles, such as the one illustrated in Figure \ref{fig-example}.

\begin{figure}[h!]
\centering\includegraphics[scale=0.375]{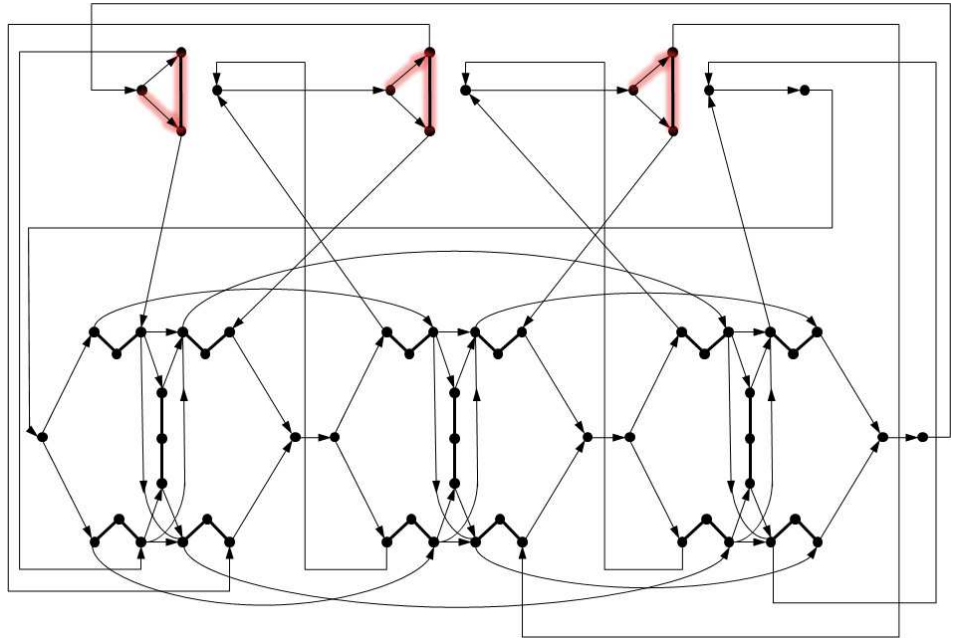}
\caption{The converted graph $\Gamma_2$ arising from the SSP instance in Example \ref{ex-ssp2}. Since the SSP instance has decision NO, the graph $\Gamma_2$ is non-Hamiltonian.\label{fig-example3}}
\end{figure}
\end{example}

\begin{acknowledgement}
The authors gratefully acknowledge useful conversations with Dr. Richard Taylor that helped in the development of this chapter.
\end{acknowledgement}

\end{document}